\documentclass{amsart}

\usepackage[english]{babel}
\usepackage{array,booktabs,float}
\usepackage{amsmath,amssymb,amsfonts}
\usepackage{amsthm}
\usepackage{bm}
\usepackage{amscd}
\usepackage{mathtools}
\usepackage{mathrsfs}
\usepackage{pb-diagram}
\usepackage{booktabs}
\usepackage{cleveref}
 
\usepackage{todonotes}
\usepackage{comment}

\theoremstyle{plain}
\newtheorem{theorem}{Theorem}[section]
\crefname{theorem}{Theorem}{Theorems}
\Crefname{theorem}{Theorem}{Theorems}
\newtheorem{proposition}[theorem]{Proposition}
\crefname{proposition}{Proposition}{Propositions}
\Crefname{proposition}{Proposition}{Propositions}
\newtheorem{lemma}[theorem]{Lemma}
\crefname{lemma}{Lemma}{Lemmas}
\Crefname{lemma}{Lemma}{Lemmas}
\newtheorem{corollary}[theorem]{Corollary}
\crefname{corollary}{Corollary}{Corollaries}
\Crefname{corollary}{Corollary}{Corollaries}

\crefname{claim}{Claim}{Claims}
\Crefname{claim}{Claim}{Claims}

\crefname{property}{Property}{Properties}
\Crefname{property}{Property}{Properties}

\crefname{problem}{Problem}{Problems}
\Crefname{problem}{Problem}{Problems}

\DeclareMathOperator{\coker}{coker}

\theoremstyle{definition}
\newtheorem{definition}[theorem]{Definition}
\crefname{definition}{Definition}{Definitions}
\Crefname{definition}{Definition}{Definitions}

\crefname{notation}{Notation}{Notations}
\Crefname{notation}{Notation}{Notations}

\crefname{convention}{Convention}{Conventions}
\Crefname{convention}{Convention}{Conventions}

\crefname{condition}{Condition}{Conditions}
\Crefname{condition}{Condition}{Conditions}

\crefname{assumption}{Assumption}{Assumptions}
\Crefname{assumption}{Assumption}{Assumptions}

\theoremstyle{remark}
\newtheorem{remark}[theorem]{Remark}
\crefname{remark}{Remark}{Remarks}
\Crefname{remark}{Remark}{Remarks}
\newtheorem{example}[theorem]{Example}
\crefname{example}{Example}{Examples}
\Crefname{example}{Example}{Examples}

\crefname{section}{Section}{Sections}
\Crefname{section}{Section}{Sections}
\crefname{subsection}{Subsection}{Subsections}
\Crefname{subsection}{Subsection}{Subsections}
\crefname{figure}{Figure}{Figures}
\Crefname{figure}{Figure}{Figures}

\newtheorem*{acknowledgement}{Acknowledgement}

\newcommand{\Z}{\mathbb{Z}}
\newcommand{\R}{\mathbb{R}}
\newcommand{\C}{\mathbb{C}}

\newcommand{\ind}{\mathrm{ind}}

\newcommand{\Spinc}{\mathrm{Spin}^c}

\begin{document}

\title{Localization of indices and orientations on $G$-instanton moduli spaces}
\author{Jin Miyazawa}
\address{Graduate School of Mathematical Sciences, the University of Tokyo, 3-8-1 Komaba, Meguro, Tokyo 153-8914, Japan}
\email{miyazawa@ms.u-tokyo.ac.jp}
 
\begin{abstract}
Joyce, Tanaka, and Upmeier give an orientation of the $G$-instanton moduli spaces on a closed four manifolds which is canonically defined using the the $\Spinc$ structure on the $4$-manifold. In this note, we describe the relation between the orientations given by other choices of the $\Spinc$ structures, in a slightly more general setting. 
Furthermore, we give an alternative proof of the orientability of the instanton moduli spaces and an alternative construction of the orientation by Joyce, Tanaka and Upmeier, by using Witten localization of the index of a Dirac type operator. 

\end{abstract}
\maketitle
\tableofcontents
\section{Introduction} 
The orientability of $G$-instanton moduli spaces and the existence of their canonical orientations are proved for the following cases: $G=\mathrm{U}(2)$ and $\mathrm{SO}(3)$ by S.~Donaldson (\cite{donaldson1987orientation}, \cite{donaldson1997geometry}), and $G=\mathrm{U}(n)$, $\mathrm{SU}(n)$ and $\mathrm{PSU}(n)$ by P.B.~Kronheimer (\cite{kronheimer2005four}). 

For general $G$-instantons, Joyce, Tanaka, and Upmeier (\cite[Theorem 4.6 (b)]{joyce2020orientations}) proved the orientability and constructed an orientation of instanton moduli spaces. 
(They actually showed the orientability and construct a canonical orientation to the moduli space of other gauge theoretic equations.) 
They used a $\Spinc$ structure on the $4$-manifold to construct the canonical orientation. 
In~\cite[Theorem 4.6 (c), Theorem 4.7 (b)]{joyce2020orientations}, they proved that their orientation was independent of the choice of $\Spinc$ structure if the gauge group $G$ was simply connected. 
Moreover, they gave other orientation of the $\mathrm{U}(n)$ instanton moduli space in~\cite[Theorem 4.6 (c)]{joyce2020orientations} using the orientation of $\mathrm{SU}(n+1)$-instanton moduli spaces. It does not depends on the choice of $\Spinc$ structure.

In this note, we describe the relation between the orientations given by other choices of $\Spinc$ structures under the condition that the group $G$ is an arbitrary compact and connected Lie group. Moreover, our setting is slightly more general than ~\cite{joyce2020orientations} since we prove the orientation theorems for associated vector bundles in the arbitrary $G$-representation $(\rho, V)$ rather than the adjoint representation. 

To prepare the proof of the main theorem, we give an analytic description of the construction of the orientation given in~\cite[Theorem 4.6 (b)]{joyce2020orientations} and we obtained a concise proof of the orientability of the moduli spaces of $G$-instantons. 

In the theorem below, the orientation given in \cite[Theorem 4.6 (b)]{joyce2020orientations} coincides with $o(P, \mathfrak g, \mathfrak s)$ for a principal $G$-bundle $P$ and a $\Spinc$ structure $\mathfrak s$. 
\begin{theorem}\label{obstruction}
    Let $P$ be a principal $G$ bundle and $(\rho, V)$ be a representation of $G$. 
    Let $l$ be a complex line bundle on $X$, and $\mathfrak s$ be a $\Spinc$ structure. 
    Let $o(P, V, \mathfrak s)$ and $o(P, V, \mathfrak s \otimes l)$ is the canonical orientations as defined in \cref{nameofori}. 
    We will denote by $w_2(E)$ the second Stiefel-Whitney class of a vector bundle $E$. 
    We have 
    \[
        o(P, V ,\mathfrak s)=(-1)^{\langle w_2(l) \cup w_2(V_{P}), [X]\rangle}o(P, V, \mathfrak{s}\otimes l).
    \]
\end{theorem}
Moreover, we determine the relation between the two orientations of the $\mathrm{U}(n)$ instanton moduli space. 

\begin{theorem}\label{Un}
    Let $P$ be a principal $\mathrm{U}(n)$ bundle and $\rho$ be the adjoint representation of $\mathrm{U}(n)$. Let $o(P, \mathfrak{u}(n), \mathrm{SU}(n+1))$ be the orientation of the $\mathrm{U}(n)$-instanton moduli space given in~\cite[Theorem 4.6 (c)]{joyce2020orientations}. Then we have 
    \[
        o(P, \mathfrak{u}(n), \mathrm{SU}(n+1))=(-1)^{\frac{1}{2}(c_1(S^+)c_1(V_{st})+c_1(V_{st})^2)- c_2(V_{st})}o(P, \mathfrak{u}(n), \mathfrak s)
    \]
    where $V_{st}$ is the associated bundle of $P$ in the standard representation of $\mathrm{U}(n)$. 
\end{theorem}
 Note that the notations $o(P, \mathfrak{s})$ and $o(P, \mathfrak{u}(n), \mathrm{SU}(n+1))$ are not appeared in~\cite{joyce2020orientations}. They are notations which used only in this note. 



\begin{acknowledgement}
    The author  would like to show his deep appreciation to his adviser Mikio Furuta for helpful suggestions, useful discussions. The author also thanks to Yoshihiro Fukumoto, Nobuo  Iida, Hokuto Konno Shinichiro Matsuo, and Masaki Taniguchi for helpful suggestions. 
    The author is supported by JSPS KAKENHI Grant Number 21J22979 and WINGS-FMSP program at the Graduate school of Mathematical Science, the University of Tokyo.
\end{acknowledgement}

\section{Settings}\label{settings}
\begin{itemize}
\item Let $(X, g)$ be an oriented closed Riemannian $4$-manifold. 
\item Let $G$ be a connected compact Lie group and $P$ be a principal $G$-bundle on $X$. Let $V$ be a finite dimensional oriented real vector space with inner product and $\rho \colon G \to \mathrm{SO}(V)$ be an orthogonal representation. We set $V_P=P \times_{\rho} V$ be the associated vector bundle. 
\item Let $\mathcal G= C^{\infty}(X, \text{Aut}(P))$ be a gauge group. We will also denote by $\rho$ the map $\text{Aut}(P) \times V_P \to V_P$ which is given by $[p, g]\cdot [p, v] \mapsto [p, \rho(g)v]$. 
\item Let $\mathcal A$ be the set of smooth connections of $P$. Let $A_0 \in \mathcal A $ be a reference connection. Let $k \ge 3$. Let $\mathcal{A}_{k+1}$ be the set $\{ A_0 + a \mid a \in L^2_{k+1}(E_0^-) \}$ and $\mathcal G_{k+2}:=\{ u \in L^2_{k+2}(\mathrm{Aut}(P))\}$. We write $\mathcal{B}_{k+1}=\mathcal{A}_{k+1}/\mathcal{G}_{k+2}$, the ambient space. 
\item Let $\mathfrak s$ be a $\Spinc$ structure on $X$ and $S=S^+\oplus S^-$ be the spinor bundle of $\mathfrak s$. We have the Clifford multiplications: 
\[
c_T \colon T^*X\otimes S^+ \to S^-, \; c_+ \colon (\R \oplus \wedge^+)\otimes S^+ \to S^+.
\]  
\item We set $E_0^+(P) :=(\R \oplus \wedge^+)\otimes V_P,\; E_0^-(P):=T^*X \otimes V_P$, $E_1^+(P) :=S^+ \otimes_{\R} V_P$ and $E_1^-(P):=S^- \otimes_{\R} V_P$. Let $E_i(P)=E_i^+(P) \oplus E_i^-(P)$. We sometimes write $E_i^{\pm}$ for $E_i^{\pm}(P)$. 
\item 
Note that $E_i$ has the Clifford action of $T^*X$. Let us fix a $\Spinc$ connection on $\mathfrak s$ and we construct an elliptic differential operator ${D_i}(A) \colon \Gamma(E_i^-) \to \Gamma(E_i^+)$ using a connection $A \in \mathcal A$ and the fixed $\Spinc$ connection. Note that ${D_0}(A)$ coincides with $d^*_A + d^+_A$ up to a scalar multiplication. 
\end{itemize}
\begin{remark}
    A $\Spinc$ structure on a $4$-manifold is determined by its spinor bundles. Thus we identified with a $\Spinc$ structure and a $\Z/2$-graded complex Clifford module $S=S^+ \oplus S^-$ with $\mathrm{rank}(S^{\pm})=2$. 
\end{remark}
\begin{definition}\label{taples}
    We define a three tuple $(D(a), \epsilon, f(b))$  as follows:
    \begin{itemize}
        \item Let $H=H^+ \oplus H^-$ be a $\Z/2$ graded real Hilbert space and $\epsilon$ be the $\Z/2$ grading operator.  
        \item Let $A$ be a compact hausdorff space and $B\subset A$ be a closed subset. 
        \item $D(a) \colon H \to H$ is an odd skew-adjoint bounded Fredholm operator which is parametrized by the point $a \in A$ continuously in the operator norm topology. 
        \item $f(b)$ is an odd skew-adjoint operator with $f(b)^2=-1$. We assume $f(b)$ is continuous in $b \in B$ in the operator norm and $D(b)f(b)+f(b)D(b)=\{D(b), f(b)\}=0$ for all $b \in B$. 
\end{itemize}
We can make an element in $KO(A, B)$ from this three tuples $(D(a), \epsilon, f(b))$ by taking the family index $\ind((D(a)+\tilde{f}(a))|_{H^+})$ where $\tilde{f}$ is an extension of $f(b)$. 
The results of Atiyah--Singer \cite{atiyah1969index} tells us that all elements in $KO(A, B)$ can be represented as above. 
If $D(a)$ is unbounded and $D'(a)=(1+D(a)D(a)^*)^{-1/2}D(a)=(1-D(a)^2)^{-1/2}D(a)$ satisfies above conditions, we write $(D'(a), \epsilon, f(b))$ by $(D(a), \epsilon, f(b))$. 
\end{definition} 

\section{Analytic description of the proof of the orientability of the moduli spaces}
We give a concise proof of the following theorem. 
\begin{theorem}\label{goal}
The determinant line bundle
\[
    \{ \det(\ker(d^*_A + d^+_A))\otimes \det(\coker(d^*_A + d^+_A))^*\}_{[A] \in \mathcal{B}_{k+1}} 
\]
 on the ambient space $\mathcal B_{k+1}$ is trivial for $k \ge 2$.  
\end{theorem}
To show \cref{goal}, we only have to prove the following proposition.  
\begin{proposition}\label{main}
Take $u \in \mathcal G_{k+2}$ and let $A_t \in \mathcal A_{k+1}$ be a $1$-parameter family of connections of $P$ such that $A_t=A_0$ if $0 \le t \le 1/4$ and $A_t=u^*A$ if $3/4 \le t \le 1$. Then the family index  $\ind (D_0(A_t)) - \ind (D_0(A_0)) \in KO([0,1],\{0,1\})$ is $0$ for $k\ge 2$ where 
 the isomorphism between the two virtual vector bundles $\ind(D_1)$ and $\ind(D_0)$ is given by $id_{\ind(D_0)}$ at $\{0\}$ and is given by $\ind(D_0) \ni \phi \mapsto \rho(u)\phi \in \ind(D_1)$ at $\{1\}$. 
\end{proposition}
From now, we give the proof of \cref{main}. 
We omit pull pack notation for a vector bundle on $X \times [0,1]$ which is a pullback of a vector bundle on $X$. 
We fix $k \ge 2$ and we write $\mathcal A, \mathcal G$ and $\mathcal B$ for $\mathcal A_{k+1}, \mathcal G_{k+2}$ and $\mathcal B_{k+1}$ respectively. 
For two operators $A$ and $B$, we denote by $\{A, B\}$ the operator $AB+BA$. 
\begin{definition}
We define a skew-adjoint elliptic differential operator $\mathcal{D}_i(t)$ which acts on $\Gamma(E_i(t)\otimes \R^2)$ by 
\[
\mathcal{D}_i(t)=\begin{pmatrix}0&D_i(A_t)\\-D^*_i(A_t)&0\end{pmatrix}\otimes \begin{pmatrix}1&0\\0&0\end{pmatrix}+ \begin{pmatrix}0&D_i(A_0)\\-D^*_i(A_0)&0\end{pmatrix} \otimes \begin{pmatrix}0&0\\0&-1\end{pmatrix} 
\]
where $D^*_i(\cdot)$ is the adjoint of $D_i(\cdot)$ with respect to the $L^2$ inner product. 
We set two isomorphisms $\epsilon_i \colon E_i \to E_i$ and $f_i(\tau) \colon E_i(\tau) \to E_i(\tau)$ ($\tau=0,1$) as
\begin{align*}
\epsilon_i&=\begin{pmatrix}1&0\\0&-1\end{pmatrix}\otimes \begin{pmatrix}1&0\\0&-1\end{pmatrix}, \\
f_i(0)&=1\otimes \begin{pmatrix}0&1\\-1&0\end{pmatrix},\;
f_i(1)= \rho(u)\otimes \begin{pmatrix}0&1\\0&0\end{pmatrix}
+\rho(u)^{-1}\otimes \begin{pmatrix}0&0\\-1&0\end{pmatrix}. 
\end{align*}
We see at once $(\epsilon_i, \mathcal{D}_i(t), (f_i(0), f_i(1)))$ define a three-tuple defined in~\cref{taples}. We write $\ind(\mathcal{D}_i)$ the element $KO([0, 1], \{0, 1 \})$ given by this three-tuple.  
\end{definition}


\begin{lemma}
The index $\ind (D_0(A_t)) - \ind (D_0(A_0)) \in KO([0,1],\{0,1\})$ in \cref{main} coincides with $\ind(\mathcal{D}_0)$. Moreover we have $\ind(\mathcal{D}_1)=0$.  
\end{lemma}
\begin{proof}
    The first half of the theorem is clear from \cref{taples}. The vector bundle $E_1$ has a complex structure and $f_1(0), f_1(1)$ are complex linear. Since $K([0,1], \{0, 1\}) \to KO([0,1], \{0, 1\})$ is the $0$ map, we have the later half of the lemma. 
\end{proof}
A slight change of the proof of the above lemma actually shows that the following Lemma. 
\begin{lemma}\label{hat}
We set $V=E_0\otimes \R^2 \oplus E_1 \otimes \R^2, V(t)=V\mid_{X \times \{t\}}$. 
We set 
\[
\hat{\mathcal{D}}(t)=\begin{pmatrix}\mathcal{D}_0(t)&0\\0&-\mathcal{D}_1(t)\end{pmatrix}
\colon \Gamma(V(t)) \to \Gamma(V(t)). 
\]
We define $\hat{\epsilon} \colon V \to V$ and $\hat f(\tau) \colon V(\tau) \to V(\tau)$ ($\tau=0,1$) by
\[
\hat{\epsilon}=\begin{pmatrix}\epsilon_0&0\\0&-\epsilon_1\end{pmatrix}, 
\;
\hat{f}(\tau)=\begin{pmatrix}f_0(\tau)&0\\0&-f_1(\tau)\end{pmatrix}. 
\]
We see at once $(\hat{\mathcal{D}}(t), \hat{\epsilon}, (\hat f(0), \hat f(1)))$ is a three-tuple defined in~\cref{taples}. 
Then we have the element $(\hat{\epsilon}, \hat{\mathcal{D}}(t), (\hat f(0), \hat f(1))) \in KO([0,1], \{0,1\})$ coincides with 
$\ind(\hat{\mathcal{D}})=\ind (D_0(A_t)) - \ind (D_0(A_0))$. 
\end{lemma}
Then we prove \cref{main}. 
\begin{proof}[proof of \cref{main}]
From \cref{hat}, we only have to prove that the index $(\hat{\mathcal{D}}(t), \hat{\epsilon}, \hat f(0), \hat f(1))$ is $0$. 
We will denote by $\sigma(\hat{\mathcal{D}}(t))$ the principal symbol of the operator $\hat{\mathcal{D}}(t)$. 
We only have to find a skew-adjoint isomorphism $s(t) : V(t) \to V(t)$ which satisfies the following conditions: 
    $s(t)$ is continuous in $t$,  $s(0)=\hat f(0)$, $s(1)=\hat f(1)$, $s^2=-1$, $\{ \hat{\epsilon}, s\}=0$ and $\{ \sigma(\hat{\mathcal{D}}(t)), s(t) \}=0$ for all $t \in [0,1]$. 
If there is such an operator, we can prove \cref{main} as follows. We set $\hat{\mathcal{D}}'(t)=(\hat{\mathcal{D}}(t)+s\hat{\mathcal{D}}(t)s)/2$. This operator is a sum of $\hat{\mathcal{D}}(t)$ and a compact operator. 
Moreover, we have $\hat{\mathcal{D}}'(\tau)=\hat{\mathcal{D}}(\tau)$ for $\tau=0,1$. 
Thus we have that the index $(\hat{\epsilon}, \hat{\mathcal{D}}'(t), \hat f(0), \hat f(1))$ coincides with the index what we want to calculate and we have $(\hat{\epsilon}, \hat{\mathcal{D}}'(t), \hat f(0), \hat f(1))=0$ since $s$ is an isomorphism and $\{\hat{\mathcal{D}}'(t), s(t) \}=0$.  

We deform the gauge transformation $u$ by a homotopy, we assume that $\rho(u)(p)=1$ for some $p \in X$. 
Let $h$ be a section of $S^+$ with $h(p)= 0$ and $h(x) \neq 0$ for $x \neq p$. 
We define the map $h: E_0  \to E_1$ by $h( \omega \otimes \xi) :=c_T(\omega)h\otimes \xi$ and $h(\eta \otimes \xi ):=c_+(\eta)h\otimes \xi$ for $\omega \in \Gamma(T^*X), \eta \in \Gamma(\R \oplus \wedge^+)$ and $\xi \in \Gamma(V_P)$. 
Then we set 
\[
\hat h=\begin{pmatrix}0& h^*\otimes id_{\R^2} \\ -h\otimes id_{\R^2} &0\end{pmatrix} \colon V \to V. 
\]
It is easy to see that $\{ \hat h , \hat{\epsilon}\}=0, \{ \hat h , \sigma(\hat{\mathcal{D}}(t))\}=0, \{ \hat h , \hat f(0)\}=0$ and $\{ \hat h , \hat f(1) \}=0$. 
We extend $\hat f(0)$ and $\hat f(1)$ to an operator on $[0,1]$ in the obvious way. 
We set the operator $s'$ by
\[
s'(t)\colon=(1-t)\hat f(0) + t\hat f(1) + t(1-t)\hat{h}. 
\]
One can check that $s'$ is a skew-adjoint operator with $\{s', \hat{\epsilon}\}=0, \{ s', \sigma(\hat{\mathcal{D}})\}=0$. $s'(\tau)=\hat{f}(\tau)$ for $\tau=0,1$. 
We show that $s'$ is an isomorphism.  
To prove this, we only have to show that at least one of $(1-t)\hat f(0) + t\hat f(1)$ or $t(1-t)\hat h$ is non-degenerate for all $(x , t) \in X \times [0,1]$ since $\{(1-t)\hat f(0) + t\hat f(1), \hat h\}=0$. 
The operator $t(1-t)\hat h$ degenerates on $\{p\} \times [0,1] \cup X \times \{0,1\} \subset X \times [0,1]$ and $(1-t)\hat f(0) + t\hat f(1)$ degenerates on $\{ x \in X \mid \rho(u)(x) \text{has}-1\text{as an  eigenvalue}\}\times \{1/2\} \subset X \times [0,1]$. 
From the assumption of $u$, we have $\rho(u)(p)=1$ therefore $(1-t)\hat f(0) + t\hat f(1)$ is non-degenerate on  $\{p\} \times [0,1] \cup X \times \{0,1\}$. 
Let $s=s'(s'(s')^{\ast})^{-1/2}=s'(-(s')^{2})^{-1/2}$ and we have $s^2=-1$. 
From the properties of $s'$, 
we have that that $s$ is an operator what we want. 
Thus we have \cref{main}. 
\end{proof}
\section{Analytic construction of a canonical orientation}
In this section, we describe an analytic construction of the canonical orientation which is given by Joyce, Tanaka, and Upmeier~\cite[Theorem 4.6(b)]{joyce2020orientations}. We continue to use the notations of \cref{settings}. 

\begin{theorem}\cite[Theorem 4.6(b)]{joyce2020orientations}\label{canonicalori}
There is a canonical orientation of the determinant line bundle on $\mathcal B$ which depends on an orientation of $H^1(X) \oplus H^+(X)$, an orientation of $V$ and the choice of $\Spinc$ structure.  
\end{theorem}
\begin{definition}\label{nameofori}
We write the orientation given in~\cref{canonicalori} by $o(P, V, \mathfrak s)$. 
\end{definition}
\begin{proof}[proof of \cref{canonicalori}]
In this proof, for a Fredholm operator $\mathcal F$, we set 
\[\det(\ind(\mathcal F)):=\wedge^{\text{max}}\ker(\mathcal F) \otimes (\wedge^{\text{max}}\coker(\mathcal F))^\ast.\] 

When a principal $G$-bundle $P_0$ is trivial, we have the trivial $P_0$ connection $\theta$.  In this case we have the natural isomorphism \[\det(\ind(D_0(\theta)))\cong \wedge^{\text{max}} (H^1(X, \R) \otimes V) \otimes (\wedge^{\text{max}}(H^0(X, \R) \oplus H^2(X, \R))\otimes V)^\ast. \] 
Thus we have that the determinant bundle has the trivialization given by the homology orientation and the orientation of $V$. 

Let $P$ be a non-trivial principal $G$-bundle on $X$. 
We take an open set $\emptyset \neq U \subset X$ such that there exist an isomorphism of the principal $G$-bundle $\tilde{\Phi} \colon P_0 |_{U} \to P|_{U}$. Using $\tilde{\Phi}$, we have a linear map $\Phi \colon V_{P_0} \to V_P$ which is isomorphism on $U$.  

From the \cref{main}, we only have to define an orientation of $\det(\ker(d^*_A + d^+_A)\otimes (\text{coker}(d^*_A + d^+_A))^*)$ for some $G$-connection $A$ and Riemannian metric $g$ on $X$. Using the complex structure of the spinor bundle of the $\Spinc$ structure $\mathfrak s$, $\det(\ker(D_1(A))\otimes \coker(D_1(A))^{\ast})$ has the canonical trivialization. 
Therefore if we take a canonical isomorphism of the determinant bundles at $[A] \in \mathcal{B}$
\begin{equation}\label{triv2}
    \det(\ind(D_0(A)) \to \det(\ind(D_1(A)))\otimes\det(\ind(D_0(\theta)))\otimes \det(\ind(D_1(\theta)))^\ast, 
\end{equation}
we have a canonical orientation of $\det(\ker(d^*_A + d^+_A)\otimes (\text{coker}(d^*_A + d^+_A))^*)$. 
Let $h'$ be a section of $S^+$ such that its zero set is in $U$, we set
\[
\hat{h}'_P=\begin{pmatrix}0& (h')^* \\ -h' &0\end{pmatrix} \colon E_0(P) \oplus E_1(P) \to E_0(P) \oplus E_1(P). 
\] 
We define $D_P(A)$ and $\epsilon_P$ as follows: 
\begin{align*}
D_P(A)&=\begin{pmatrix} 0&D_0(A)\\-D^*_0(A)&0 \end{pmatrix} \oplus \begin{pmatrix} 0&-D_1(A)\\D^*_1(A)&0 \end{pmatrix} , \;
\epsilon_P =\begin{pmatrix}1&0\\0&-1\end{pmatrix}\oplus \begin{pmatrix}-1&0\\0&1\end{pmatrix}  \\ \colon &\Gamma(E_0(P) \oplus E_1(P)) \to \Gamma(E_0(P) \oplus E_1(P)).
\end{align*}
It easy to see that $\{\epsilon_P , \hat{h}'_P \}=0$ and $\{ \sigma(D_P(A)), \hat{h}'\}=0$ where $\sigma(D_P(A))$ is the principal symbol of $D_P(A)$. 
Let us denote by $D_{P, P_0}$ and $\epsilon$ the operators 
\begin{equation}\label{DPP}
    D_{P, P_0}= D_P(A) \oplus(- D_{P_0}(\theta)) 
\end{equation}
and $\epsilon_P \oplus (-\epsilon_{P_0})$ respectively. We set 
\[
    \hat{\Phi}\colon (E_0(P) \oplus E_1(P))\oplus (E_0(P_0) \oplus E_1(P_0)) \to (E_0(P) \oplus E_1(P))\oplus (E_0(P_0) \oplus E_1(P_0)).
\]
by 
\begin{equation}\label{hatphi}
    \hat{\Phi} = \begin{pmatrix}0&-(\mathrm{id}_{T^*X \oplus\R \oplus \wedge^+\oplus S^+ \oplus S^-} )\otimes \Phi^{\ast}\\(\mathrm{id}_{T^*X \oplus\R \oplus \wedge^+\oplus S^+ \oplus S^-} ) \otimes \Phi&0\end{pmatrix}. 
\end{equation}   
We will denote by $\hat{h'}$ the operator $\hat{h'}_P \oplus (-\hat{h'}_{P_0})$. 
It is easy to see that $D_{P, P_0}$, $\hat{\Phi}$ and $\hat{h'}$ are skew-adjoint and $\sigma(D_{P, P_0})$, $\hat{\Phi}$, and $\hat{h'}$ anti-commute each other and anti-commute with $\epsilon$. 
From the definition of $h'$ and $\Phi$, the operator $s=\hat{h'}+\hat{\Phi}$ is non-degenerate. 
Thus we have that $D_{P, P_0}+ms$ is an isomorphism for large $m \ge 0$. Using the homotopy of the operator with parameter $m$, we can take a trivialization 
\begin{equation}\label{triv1}
   T(h', \Phi) \colon 
   \det(\ind(D_{P, P_0}))=\det(\ind(D_0(A) \oplus D_1(A))) \otimes \det(\ind(D_0(\theta) \oplus D_1(\theta)))^\ast \to \R. 
\end{equation}
Thus we have an isomorphism \eqref{triv2} using the homomorphism $\Phi$, the $\Spinc$ structure $\mathfrak s$, and the section $h'$ of the positive spinor bundle of $\mathfrak s$. 


Let us prove that the orientation given by the map $T(h', \Phi)$ is independent of the choice of $(h', \Phi)$. If we take two pairs $(h'_0, \Phi_0)$ and $(h'_1, \Phi_1)$, we can take a homotopy $H(t)$ from $h'_0$ to $h'_1$ and a homotopy $\Phi(t)$ from $\Phi_0$ to $\Phi_1$ such that $\hat{h}'(t)+\hat{\Phi}(t)$ is an isomorphism for all $t \in [0, 1]$. Thus we have the homotopy of two trivialization and we have two trivialization of $\det(\ind(D))$ coincides. 
\end{proof}

If $X$ has  an almost complex structure $J$, there is the $\Spinc$ structure $\mathfrak s_J$ whose spinor bundle $S_J$ is $T^{0, 1}X \oplus (\underline{\C} \oplus K^{-1})$. The Clifford actions $c_T$ and $c_+$ give an isomorphism 
\begin{equation}
    T^*X \oplus (\R \oplus \wedge^+) \to T^{0, 1}X \oplus \underline{\C} \oplus K^{-1}
\end{equation}
by $(\xi, \omega) \to (c_T(\xi)\cdot 1, c_+(\omega)\cdot 1)$. 
Thus we have a complex structure on $T^*X \oplus (\R \oplus \wedge^+)$ through this isomorphism. We have the canonical orientation of the determinant bundle on $\mathcal B$ by using this almost complex structure (see~\cite[Section 3, (c)]{donaldson1987orientation} or~\cite[Theorem 2.5]{joyce2020orientations}). 

\begin{proposition}\label{complex}
If $X$ has  an almost complex structure $J$, there is a canonical $\Spinc$ structure $\mathfrak s_J$ such that our canonical orientation $o(P, V, \mathfrak s_J)$ coincides with the canonical orientation given by the almost complex structure $J$.  
\end{proposition}
\begin{proof}
To construct the trivialization \eqref{triv1}, we can take a non-vanishing section $h'$ to be $1 \in \Gamma(\C)$. 
From the definition of the complex structure on $T^*X \oplus (\R \oplus \wedge^+)$ given by $J$, the map \eqref{triv1} coincides with the trivialization of $\det(\ind(D_{P, P_0}))$ given by the complex structure on $T^*X \oplus (\R \oplus \wedge^+)$. 
\end{proof}

\begin{proposition}
When the representation $\rho$ is the adjoint representation, our canonical orientation coincides with the orientation given in~\cite[Theorem 4.6 (b)]{joyce2020orientations}. 
\end{proposition}
\begin{proof}

Let us take an open covering $X=U \cup V$ such that $V$ is not empty and the $\hat{\Phi}$ is isomorphism on $V$. Let $\{\rho_U, \rho_V \}$ be a partition of unity with respect to the open covering. 
We denote by $D$ the operator $D_{P, P_0}$ which acts on $\Gamma(E_0(P)\oplus E_1(P) \oplus E_0(P_0) \oplus E_1(P_0))$ given in the proof of \cref{canonicalori}, \eqref{DPP}. 
Set $\tilde D= (1+DD^*)^{-1/2}D=(1-D^2)^{-1/2}D$, a $0$-th order pseudo differential operator.  
We deform continuously the operator $\tilde D$ 
to the $0$-th order elliptic pseudo differential operator $\rho_U \tilde{D} \rho_U + (1-\rho_U^2)\hat{\Phi}$ since the principal symbol of $\tilde D$ anti-commutes with $\hat{\Phi}$. 
The kernel of the operator $\rho_U \tilde{D} \rho_U + (1-\rho_U^2)\hat{\Phi}$ is localized to the open set $U$. 
We have the non-vanishing section $h$ of the positive spinor bundle on $U$. We have an almost complex structure $J$ on $U$ such that $S^+ |_U \cong \underline{\C} \oplus \wedge^{0,2} \cong \underline{\C} \oplus K^{-1}$ and $h$ coincides the section $1 \in \Gamma(\underline{\C})$ on $U$. In~\cite[Theorem 4.6 (b)]{joyce2020orientations}, they give an orientation on $U$ by constructing the almost complex structure $J$. 
From the construction of the complex structure of $T^*X \oplus (\R \oplus \wedge^+)$ induced by $J$, the $\hat{h}$ is complex linear. The operator $\rho_U \tilde{D} \rho_U + (1-\rho_U^2)\hat{\Phi}+  \hat{h}$ is an isomorphism. We have a homotopy of the two isomorphisms between $\rho_U \tilde{D} \rho_U + (1-\rho_U^2)\hat{\Phi} + \hat{h}$ to $(1+DD^*)^{-1/2}(D + \hat{\Phi} + \hat{h})$. Thus our orientation coincides with the orientation given in~\cite[Theorem 4.6 (b)]{joyce2020orientations}. 
\end{proof}

\section{The dependence of the Canonical orientation}
We see how the canonical orientation $o(P, V, \mathfrak s)$ change if we take an other $\Spinc$ structure. 

In this section, let us denote by $E_1(P, \mathfrak s)$ the vector bundle $E_1(P)$ defined in \cref{settings} since this vector bundle depends on the $\Spinc$ structure $\mathfrak s$. 
First, we give an example that the orientation $o(P, V, \mathfrak s)$ is reversed when we take an other $\Spinc$ structure. 

\begin{example}\label{example}
    Let $X=\C P^2$ and $G=\mathrm{SO}(3)$. 
    We will denote by $\gamma$ a tautological bundle of $\C P^2$ and $P_\gamma$ the oriented orthonormal frame bundle of $E=\underline{\R} \oplus \gamma$. Note that $E$ is the adjoint bundle of $P_\gamma$ in the standard representation $\rho_{st}$ of $\mathrm{SO}(3)$. 
    Let $K$ be the canonical line bundle of $\C P^2$. We have the $\Spinc$ structure $\mathfrak s_{\text{can}}$ which is given by setting $S^+=\underline{\C} \oplus K^{-1}$ and $S^-=T^{0, 1}X$ as a $\Z/2$-graded complex Clifford module. 
    
    Let $A$ be a $G$-connection on $P_\gamma$. Let us show that 
    \[o(P_\gamma,\R^3, \mathfrak{s}_{can})=-o(P_\gamma,\R^3, \mathfrak{s}_{can}\otimes K).
    \] 
    Recall that the orientation 
    \begin{equation}\label{complexori}
        \det(\ind_{\R}(D_1(A, \mathfrak{s}_{\text{can}}))\otimes \det(\ind_{\R}(D_1(\theta, \mathfrak{s}_{\text{can}})))^\ast \to \R
    \end{equation}
    is given by the complex structure of the kernel and the cokernel of the operator $D_1(A, \mathfrak{s}_{\text{can}}) \oplus D_1(\theta, \mathfrak{s}_{\text{can}})$. 
    If we reverse the complex structure of $S=S^+ \oplus S^-$, the orientation of the determinant \eqref{complexori} is reversed since the index $\ind_{\C}(D_1(A, \mathfrak{s}_{\text{can}})-\ind_{\C}(D_1(\theta, \mathfrak{s}_{\text{can}}))=p_1(V_{P_\gamma})=1$ is odd. 
    Therefore, the orientation $o(P_\gamma,\R^3, \mathfrak{s}_{can})$ is reversed since the isomorphism $T(h', \Phi)$ is independent of the choice of complex structure of $S$. 
    The spinor bundle of the $\Spinc$ structure $\mathfrak{s}_{\text{can}} \otimes K$ coincides with $\overline{S^+} \oplus \overline{S^-}$ as an complex Clifford module on $\C P^2$. 
    Thus we have $o(P_\gamma, \R^3, \mathfrak{s}_{can})=-o(P_\gamma, \R^3, \mathfrak{s}_{can}\otimes K)$.  
\end{example}


\begin{definition}
    Let $\mathfrak s$ be a $\Spinc$ structure on $X$ and $L$ be a complex line bundle on $X$. Let $P_0=X \times G$ and let $P$ be an arbitrary principal $G$ bundle on $X$. We will denote by $D_{P, \mathfrak s}$ the operator $D_{P, P_0}$ defined in the proof of \cref{canonicalori}, \eqref{DPP}. 
    Let $v \in \Gamma(L)$ be an section. 
    We will denote by $V(P_0, P, \mathfrak{s})$ the vector bundle $(E_0(P)\oplus E_1(P, \mathfrak{s}))\oplus (E_0(P_0)\oplus E_1(P_0, \mathfrak{s}))$. 
    We set the following operators: 
    \begin{align*}
        &\tilde{D}_{P, \mathfrak s, L} = 
            \begin{pmatrix}
                D_{P, \mathfrak s} & 0 \\ 0 & -D_{P, \mathfrak{s}\otimes L}
            \end{pmatrix}, \; 
            \tilde{\epsilon}=\begin{pmatrix}
                1 & 0 \\ 0 & -1
            \end{pmatrix}
              \\
              \colon&\Gamma(V(P_0, P, \mathfrak{s}))\oplus(V(P_0, P, \mathfrak{s}\otimes L))) \to 
            \Gamma(V(P_0, P, \mathfrak{s}))\oplus(V(P_0, P, \mathfrak{s}\otimes L))). 
    \end{align*}
    For $v \in \Gamma(L)$, we define the following operators:
    \[
        \otimes v \colon \Gamma(E_1(P_0, \mathfrak{s})\oplus E_1(P, \mathfrak{s}))  \to \Gamma(E_1(P_0, \mathfrak{s}\otimes L)\oplus E_1(P, \mathfrak{s}\otimes L))  
    \]
    and 
    \begin{align*}
        &\hat{v}=\begin{pmatrix}
            0 & -(id_{E_0(P_0)\oplus E_0(P)} \oplus (\otimes v)^\ast) \\ id_{E_0(P_0)\oplus E_0(P)} \oplus (\otimes v) & 0
        \end{pmatrix} \\
        \colon &\Gamma(V(P_0, P, \mathfrak{s})\oplus V(P_0, P, \mathfrak{s}\otimes L)) \to 
            \Gamma(V(P_0, P, \mathfrak{s})\oplus V(P_0, P, \mathfrak{s}\otimes L)).
    \end{align*} 
    It is easy to see $\{\sigma(\tilde{D}_{P, \mathfrak s, L}), \hat{v}\}=0$, and $\tilde{D}_{P, \mathfrak s, L}$ and $\hat{v}$ are skew-symmetric and $\tilde{\epsilon}$ is symmetric. 
\end{definition}

We now prove the main theorem. 

\begin{proof}[Proof of \cref{obstruction}]
    \underline{Step.1} We show that if $w_2(l) \cup w_2(V_{P})=0$ we have $o(P, V, \mathfrak s)=o(P, V, \mathfrak{s}\otimes l)$. 
    Let $v$ be a transverse section of the complex line bundle $l$. 
    We set $\Sigma=v^{-1}(0)$ and $N(\Sigma)$ be a tubular neighborhood of $\Sigma$. 

    Let $\Phi \colon V_{P_0} \to V_P$ be a homomorphism. 
    If necessary, taking direct sum of the trivial $\R^n$ bundle with $V_{P_0}$ and $V_P$, we may assume that $\dim_{\R}(V_{P_0})=\dim_{\R}(V_P) \ge 3$. We take $\Phi$ to be isomorphic on $N(\Sigma)$ since we assume $w_2(l) \cup w_2(V_{P})=0$. 
    
    We identify $N(\Sigma)$ with the unit disc bundle of the normal bundle of $\Sigma$. Let $N_{r}(\Sigma) \subset N(\Sigma)$ be a disc bundle with the radius $r \in (0,1]$. 
    We assume that $\lvert v \rvert=1$ on $X\setminus N_{1/2}(\Sigma)$. 
    We take a section $h \in \Gamma(S^+)$ such that $h=0$ on $N_{3/4}(\Sigma)$ and $h \neq 0$ on $X \setminus N(\Sigma)$. We set $h'=h\otimes v \in \Gamma(S^+ \otimes l)$. Note that $h=h' \otimes v^\ast$. 
    
    In this setting, we have that the operators $\hat{v}$, $\hat{\Phi}$ and 
    \begin{align*}
        \hat{h}\oplus(-\hat{h'})
        \colon \Gamma(V(P_0, P, \mathfrak{s})\oplus V(P_0, P, \mathfrak{s}\otimes L)) \to 
            \Gamma(V(P_0, P, \mathfrak{s}) \oplus V(P_0, P, \mathfrak{s}\otimes L))
    \end{align*}
    are anti-commutative. From the construction of these operators, we have a family of operators
    \[
        t(\hat{h}\oplus(-\hat{h'})) + (1-t)\hat{v}+ (\hat{\Phi}\oplus(-\hat{\Phi}))
    \] 
    and we have that these are isomorphism for $t \in [0, 1]$. This gives a homotopy of the two trivializations of the determinant $\det(\ind(D_0(A)))$. 

    \underline{Step.2} 
    We show that if $w_2(V_P)\cup w_2(l)=1$ we have $o(P,V, \mathfrak s)=-o(P, V, \mathfrak{s}\otimes l)$. 
    Let $P_{0, X}=X \times \mathrm{SO}(3)$ and $P_{0, \C P^2}=\C P^2 \times G$. 
    We set two principal $G \times \mathrm{SO}(3)$ bundles 
    \[
        P_X=P \times_X P_{0, X}, \; P_{\C P^2}= P_{0, \C P^2} \times_{\C P^2} P_\gamma
    \]
    where $P_\gamma$ is a principal $\mathrm{SO}(3)$ bundle on $\C P^2$ which appears in \cref{example}. Let $\rho$ is the representation of $G$ such that $P \times_{\rho} V =V_P$ and $\rho_{st}$ be the standard representation of $\mathrm{SO}(3)$. We set $V_X$ is the associated vector bundle of $P_X$ in the representation $\rho \times \rho_{st}$ of $G \times \mathrm{SO}(3)$ and we set $V_{\C P^2}$ is that of $P_{\C P^2}$. We take a four-tuple $(B, U, \Phi, h)$ for $(X, \mathfrak s, P_X)$ as follows:
    \begin{itemize}
        \item Let $U$ be an open set in $X$ and $B$ is an open set which is diffeomorphic to the unit ball in $\R^4$ and its closure $\bar{B}$ is a subset of $U$. 
        \item $\Phi \colon V_X \to V_0$ is a linear map which is isomorphic on $U$. 
        \item $h$ is a section of the positive spinor of the $\Spinc$ structure $\mathfrak s$ and which is non-zero on the set $X \setminus U$ and $h=0$ on a neighborhood of $\bar{B}$. 
    \end{itemize}
    We take a four-tuple $(B', U', \Phi', h')$ for $(\C P^2, \mathfrak s_{can}, P_{\C P^2})$ as well. 
    We take the connected sum $X\# \C P^2=(X\setminus B) \cup_{\partial B=\partial B'} (\C P^2 \setminus B')$ and let $\mathcal P =P_X \# P_{\C P^2}$. Let $A$ be a connection of $P_X$ which is flat on $U$ and $A'$ be a connection of $P_{\C P^2}$ which is flat on $U'$. 
    We will denote by $\tilde{A}$ the glued connection $A\# A'$ on $\mathcal P$. We take the glued $\Spinc$ structure $\tilde{\mathfrak{s}}=\mathfrak s \# \mathfrak s_{can}$ on $X\# \C P^2$. 
    Let $V_{\mathcal{P}}$ be the associated vector bundle of $\mathcal P$ in the representation $\rho \times \rho_{st}$. 
    We have glued map $\tilde{\Phi}=\Phi \# \Phi'$ from $V_{\mathcal{P}}$ to the trivial vector bundle. The map $\tilde{\Phi}$ is isomorphism on the open set $\tilde U=U\setminus B \cup_{\partial B=\partial B'} U' \setminus B' \subset X\# \C P^2$. We also have the glued section $\tilde{h}=h \# h'$. Note that $\tilde{\Phi}$ and $\tilde{h}$ will not vanish in a same point on $X \# \C P^2$. 
    For an elliptic operator $D$, we define an elliptic pseudo-differential operator $\tilde D$ as $(1+D^*D)^{-1/2}D$. 
    Let $D_{\mathcal P, \mathcal P_0}$ be the operator defined in the proof of~\cref{canonicalori}, \eqref{DPP} and let $\hat{\tilde{\Phi}}$ be the map defined in the proof of~\cref{canonicalori}, \eqref{hatphi} using the map $\tilde{\Phi}$. 
    We set $(\rho_{\tilde U}, 1-\rho_{\tilde U})$ be a partition of unity associated to the open covering $X\# \C P^2=\tilde U \cup (X\# \C P^2 \setminus \partial B)$.  
    There is a homotopy between two elliptic operators $\tilde D_{\mathcal P, \mathcal P_0}$ and $\rho_{\tilde U} (\tilde D_{\mathcal P, \mathcal P_0})\rho_{\tilde U}+ (1-\rho_{\tilde U}^2)\hat{\tilde{\Phi}}$ and one can check that $\ker(\rho_{\tilde U} (\tilde D_{\mathcal P, \mathcal P_0})\rho_{\tilde U}+ (1-\rho_{\tilde U}^2)\hat{\tilde{\Phi}})=\ker(\rho_\mathrm{U}(\tilde{D}_{P_X, P_0})\rho_U+ (1-\rho_U^2)\hat{\Phi}) \oplus \ker(\rho_{U'}(\tilde{D}_{P_{\C P^2}, P_0})\rho_{U'}+ (1-\rho_{U'}^2)\hat{\Phi'})$. 
    This homotopy of the elliptic operators induces the following two isomorphisms
    \begin{align*}
    &F_{\tilde{\Phi}} \colon \det(\ind(D_0(\tilde A))\oplus(-D_0(\theta)) ) \to \\
        &\det(\ind(D_0(A))\oplus(-D_0(\theta)) )\otimes \det(\ind(D_0(A'))\oplus(-D_0(\theta)) )
    \end{align*}
    and 
    \begin{align*}
        F_{\tilde{\mathfrak s}} \colon \det(\ind(D_{\mathcal P, \mathcal P_0})) \to \det(\ind(D_{P_X, P_0}))\otimes \det(\ind(D_{P_{\C P^2}, P_0})). 
    \end{align*}
    We see that $F_{\tilde{\mathfrak s}} |_{E_0(\mathcal P) \oplus E_0(\mathcal P_0)}=F_{\tilde{\Phi}}$ and $T(\tilde{h}, \tilde{\Phi})=(T(h, \Phi) \otimes T(h', \Phi') )\circ F_{\tilde{\mathfrak s}}$. Thus we have $F_{\tilde{\Phi}}(o(\mathcal P, \mathfrak{s}\#\mathfrak{s}_{\text{can}}))=o(P_X, \mathfrak s) o(P_{\C P^2}, \mathfrak{s}_{\text{can}})$.
    
    In the same way, we have $F_{\tilde{\Phi}}(o(\mathcal P, \mathfrak{s}\#\mathfrak{s}_{\text{can}}\otimes l\#K))=o(P_X, \mathfrak s \otimes l) o(P_{\C P^2}, \mathfrak{s}_{\text{can}}\otimes K)$. 
    From the definition of the vector bundle $V_{\mathcal{P}}$, we have 
    \[
        w_2(l\# K) \cup w_2(V_{\mathcal{P}})=w_2(l) \cup w_2(V_P) + w_2(K) \cup w_2(V_{P_\gamma})=1+1=0. 
    \]
    Thus we have $o(\mathcal P, \mathfrak{s}\#\mathfrak{s}_{\text{can}}\otimes l\#K)=o(\mathcal P, \mathfrak{s}\#\mathfrak{s}_{\text{can}})$ from \underline{Step 1.}. On the other hand, from~\cref{example}, we have $o(P_{\C P^2}, \mathfrak{s}_{\text{can}}\otimes K)=-o(P_{\C P^2}, \mathfrak{s}_{\text{can}})$. Thus we have $o(P_X, \mathfrak s \otimes l)=-o(P_X, \mathfrak s)$. 
\end{proof}

\begin{corollary}
    If the map $\pi_1(G) \to \pi_1(\mathrm{SO}(V))$ induced by the representation $\rho$ is trivial, $o(P,V, \mathfrak s)$ is independent of the $\Spinc$ structure $\mathfrak s$ on $X$. Especially, if $G$ is simply connected, 
    $o(P,V, \mathfrak s)$ is independent of $\mathfrak s$. 
\end{corollary}
\begin{proof}
    Let $f$ be the composition of the map $H^2(X, \pi_1(G)) \to H^2(X, \pi_1(\mathrm{SO}(V))) \to H^2(X, \Z/2)$ where the first map is induced by the map $\pi_1(G) \to \pi_1(\mathrm{SO}(V))$ and the second map is $\bmod  2$ reduction. Note that if $\dim V \ge 3$, the second map is an isomorphism. 
    The obstruction $c$ of the trivialization of the principal $G$-bundle $P$ takes value in $H^2(X, \pi_1(G))$. The second Stiefel-Whitney class of $V_P$ is give by $f(c)$. Thus we have the conclusion. 
\end{proof}
If $(\rho, V)$ is a complex representation with respect to the complex structure $J$ on $V$, we have the orientation $o(P, V, J)$ of $\det(\ind(D_0(A)))$ given by the complex structure. In this setting we have the following proposition. 

\begin{proposition}\label{complexandcanori}
    Let $x=\langle (c_1(S^+)c_1(V_P)+c_1(V_P)^2)/2 - c_2(V_P), [X] \rangle$. Then we have $o(P, V, J)o(P, V, \mathfrak s)=(-1)^x$. 
\end{proposition}
\begin{proof}[Proof of \cref{Un}]
    The isomorphism \eqref{triv1}:
    \[
      T(h', \Phi) \colon  \det(\ind(D_0(A) \oplus D_1(A))) \otimes \det(\ind(D_0(\theta) \oplus D_1(\theta)))^\ast \to \R
    \]
    given in the proof of \cref{canonicalori}, coincides with the orientation given by the complex structure of $V_P$ since the maps $\hat{h'}$ and $\hat{\Phi}$ are complex linear with the complex structure of $E_0(P) \oplus E_0(P_0) \oplus E_1(P) \oplus E_1(P_0)$ induced by $J$. 
    The orientation $o(P, V, \mathfrak s)$ is given by trivialize $\det(\ind(D_1(A))\otimes \det(D_1(\theta)))^\ast$ using the complex structure $J_S$ of $S=S^+ \oplus S^-$ and $\det(D_0(\theta))$ trivialize using the homology orientation. 
    If we give a complex structure $S \otimes_{\R} (V_P \oplus V)$ using the complex structure $J$, we see that there is a natural isomorphism between the complex vector bundles $S \otimes_{\R} (V_P \oplus V) \cong \overline{S}\otimes_{\C} (V_P \oplus V) \oplus S\otimes_{\C} (V_P \oplus V)$. 
    If the complex index of the Dirac operator $D$ on $\overline{S}\otimes_{\C}(V_P \oplus V)$ is odd, the trivialization of $\det(\ind(D_1(A))\otimes \det(D_1(\theta)))^\ast$ given by $J_S$ is different from the trivialization given by $J$. 
    Thus we have $o(P, V, J)o(P, V, \mathfrak s)=1$ if and only if $\ind_{\C}(D)\equiv 0 \mod 2$. The proposition follows from the Atiyah-Singer index theorem. 
\end{proof}

We can now prove \cref{Un}. 

\begin{proof}[Proof of \cref{Un}]
    The orientation given in~\cite[Theorem 4.6 (c)]{joyce2020orientations} is defined by the complex structure given by the standard representation of $\mathrm{U}(n)$ and the canonical orientation  of the $\mathrm{SU}(n+1)$-instanton moduli space, which is not depend on the choice of $\Spinc$ structure. The conclusion follows from~\cref{complexandcanori}. 
\end{proof}

\bibliographystyle{jplain}
\bibliography{bibidatabase}

\end{document}